\documentclass[a4paper]{amsart}

\usepackage{amsmath,amsfonts,amssymb}
\usepackage{amsthm}
\usepackage{ifthen}
\usepackage{array}
\usepackage{url}
\usepackage{multirow}
\usepackage{graphicx}
\usepackage{enumitem}

\usepackage{amssymb}
\usepackage{graphicx}

\usepackage{url}
\usepackage{tabularx}
\usepackage{longtable}

\usepackage{pdflscape}

\RequirePackage[T1]{fontenc}
\RequirePackage[utf8]{inputenc}

\usepackage{soul} 
\usepackage[usenames,dvipsnames,svgnames,table]{xcolor}  
\usepackage{textcase}

\theoremstyle{plain}
\newtheorem{theorem}{Theorem}[section]

\theoremstyle{definition}

\newtheorem{definition}[theorem]{Definition}

\newcommand{\binomial}[2]{\left(\begin{array}{c}#1\\#2\end{array}\right)}

\newcommand{\ifc}{{\rm if \ }}

\newcommand{\boxtensor}{{\Box\kern-9.03pt\raise1.42pt\hbox{$\times$}}}

\newcommand{\sF}{{\mathcal F}}

\newcommand{\F}{{\mathbb{F}}}

\newcommand{\be}{\begin{eqnarray}}
\newcommand{\ee}{\end{eqnarray}}
\newcommand{\nn}{{\nonumber}}
\newcommand{\dd}{\displaystyle}

\begin{document}

\title[Family complexity and cross-correlation measure]{Family complexity and cross-correlation measure for families of binary sequences}

\author[A.\ Winterhof]{Arne Winterhof}
\author[O.\ Yayla]{O\u{g}uz Yayla}

\address[Arne Winterhof, O\u{g}uz Yayla]{Johann Radon Institute for Computational and Applied Mathematics (RICAM)\\
Austrian Academy of Sciences\\
Altenbergerstr. 69\\
4040 Linz\\Austria}
\email{arne.winterhof@ricam.oeaw.ac.at}
\email{oguz.yayla@ricam.oeaw.ac.at}

\begin{abstract}
We study the relationship between two measures of pseudorandomness for families of binary sequences: family complexity and cross-correlation measure introduced by Ahlswede et al.\ in 2003 and 
recently by Gyarmati et al., respectively. 
More precisely, we estimate the family complexity of a family $(e_{i,1},\ldots,e_{i,N})\in \{-1,+1\}^N$, $i=1,\ldots,F$, of binary sequences of length $N$ in terms of the   
cross-correlation measure of its dual family $(e_{1,n},\ldots,e_{F,n})\in \{-1,+1\}^F$, $n=1,\ldots,N$. 

We apply this result to the family of sequences of 
Legendre symbols with irreducible quadratic polynomials modulo $p$ 
with middle coefficient $0$, that is, 
$e_{i,n}=\left(\frac{n^2-bi^2}{p}\right)_{n=1}^{(p-1)/2}$ for $i=1,\ldots,(p-1)/2$, where $b$ is a quadratic nonresidue modulo $p$, showing that this family as well as its dual family have
both a large family complexity and a small cross-correlation measure up to a rather large order. 
\end{abstract}

\maketitle

\section{Introduction}

Pseudorandom binary sequences are used in many areas such as telecommunication, cryptography, simulation, spectroscopy, see for example \cite{GG2005,Gya2013,TW2007}. 
The quality of a pseudorandom sequence has to be screened by statistical test packages 
(for example L'Ecuyer's TESTU01, Marsaglia's Diehard or the NIST battery) as well as by proving theoretical results on certain measures of pseudorandomness such as the 
correlation measure of order $\ell$ introduced by Mauduit and S\'ark\"ozy \cite{MS1997}. Here we focus on theoretical results.

In many applications such as cryptography one needs a large family of good pseudorandom sequences and has to prove bounds on several figures of merit \cite{Gya2013,GMS2014}. 
In this paper we study the relationship of two such quality measures, the family complexity, short $f$-complexity, and the cross-correlation measure of order $\ell$ of families of binary sequences. 

Ahlswede et al.~\cite{AKMS2003} introduced the $f$-complexity as follows.

\begin{definition}
The \textit{$f$-complexity} $C(\sF)$ of a family $\sF$ of binary sequences $E_N \in \{-1,+1\}^N$ of length $N$ 
is the greatest integer $j \geq 0$ such that for any $1 \leq i_1 < i_2< \cdots < i_j \leq N$ and any $\epsilon_1,\epsilon_2, \ldots,  \epsilon_j \in \{-1,+1\}$ 
there is a sequence $E_N = \{e_1,e_2,\ldots , e_N\}\in \sF$ with $$e_{i_1}=\epsilon_1,e_{i_2}=\epsilon_2, \ldots ,e_{i_j}=\epsilon_j.$$
\end{definition}

We have the trivial upper bound 
\be \label{eqn.bound_f}
2^{C(\sF)} \leq |\sF|,
\ee
where $|\sF|$ denotes the size of the family $\sF$.

Gyarmati et al.~\cite{GMS2014} introduced the cross-correlation measure of order $\ell$.


\begin{definition} \label{def.ccm}
The \textit{cross-correlation measure of order $\ell$} of a family $\sF$ of binary sequences $E_{i,N} = (e_{i,1},e_{i,2},\ldots , e_{i,N}) \in \{-1+1\}^N$,  $i=1,2, \ldots , F$, 
is defined as 
$$
\Phi_\ell(\sF) = \max_{M,D,I}\left| \sum_{n=1}^{M}{e_{i_1,n+d_1} \cdots e_{i_\ell,n+d_\ell}}\right|,
$$
where $D$ denotes an $\ell$ tuple $(d_1,d_2,\ldots , d_\ell)$ of integers such that $0 \leq d_1 \leq d_2 \leq \cdots \leq d_\ell < M+d_\ell \leq N$ and $d_i \neq d_j$ 
if $E_{i,N} = E_{j,N}$ for $i \neq j$, and $I$ denotes an $\ell$ tuple $(i_1,i_2, \ldots , i_\ell)\in\{1,2,\ldots ,F\}^\ell$.
\end{definition}

In Section \ref{sec.ccm} we estimate the $f$-complexity of a family 
of binary sequences 
$$E_{i,N} = (e_{i,1},e_{i,2},\ldots , e_{i,N}) \in \{-1+1\}^N, \quad i=1,\ldots,F,$$
in terms of the cross-correlation measure of the {\em dual family} $\overline{\sF}$ of binary sequences 
$$E_{i,F} = (e_{1,i},e_{2,i},\ldots , e_{F,i}) \in \{-1+1\}^{F}, \quad i=1,\ldots,N.$$ 

In Section~\ref{sec.example} we apply this result to prove a bound on the $f$-complexity of the family of sequences of Legendre-symbols of monic quadratic irreducible polynomials 
with vanishing middle coefficient
$$\sF=\left\{\left(\frac{n^2-bi^2}{p}\right)_{n=1}^{(p-1)/2} : i=1,\ldots,(p-1)/2\right\},$$ 
where $b$ is a quadratic non-residue modulo $p$, 
showing that this family as well as its dual family have
both a large family complexity and a small cross-correlation measure up to a rather large order. 

We note that there are several related constructions of families of binary sequences defined with the Legendre symbol and polynomials, see \cite{Gya2009, GMS2014} and references therein. For instance, the family given in \cite{Gya2009} has a large family complexity but a large cross-correlation measure of order 2. Moreover, the families given in \cite{GMS2014}  have a small cross-correlation measure but it is not easy to measure their family complexity, for further details see the remarks in Section \ref{sec.example}. One of the aims of this study is to construct a family of binary sequences having both a large family complexity and a small cross-correlation measure. 

Throughout the paper, the notation $U \ll V$ is equivalent to the statement that $\vert U \vert \leq cV$ holds with some positive constant $c$. Moreover, the notation $f(n) = o(1)$ is equivalent to $$\lim_{n \to \infty}{f(n)} = 0.$$ 

\section{A relation between family complexity and cross-correlation measure}\label{sec.ccm}
In this section we prove the following relationship between the $f$-complexity of a family of binary sequences and the cross-correlation measure of its dual family.

\begin{theorem}
\label{thm.f_ccm}
Let $\sF$ be a family of binary sequences $(e_{k,1},\ldots ,e_{k,N}) \in \{-1,+1\}^N$ for $k=1,2,\ldots ,F$ and $\overline{\sF}$ its dual family of binary sequences 
$(e_{1,n},e_{2,n},\ldots ,e_{F,n}) \in \{-1,+1\}^{F}$ for $n=1,2,\ldots ,N$. Then we have
$$
C(\sF) \geq \left\lceil \log_2{F} - \log_2{\max_{1 \leq i \leq \log_2{F}}{\Phi_{i}(\overline{\sF})}} \right\rceil -1,
$$
where $\log_2$ denotes the binary logarithm.
\end{theorem}
\begin{proof}
We are looking for the largest $j$ such that any specification 
\be \label{eqn.spec}
e_{k,n_1}=\epsilon_1,e_{k,n_2}=\epsilon_2,\ldots ,e_{k,n_j}=\epsilon_j
\ee
occurs in the family $\sF$ for some $k \in \{1,2,\ldots ,F\}$.
We will prove a sufficient condition for the existence of a sequence in $\sF$ satisfying (\ref{eqn.spec}) by a counting argument. 
We know that
$$\frac{1}{2}(1+\epsilon_ie_{k,n_i}) = \left\lbrace \begin{array}{ll}1 & \ifc e_{k,n_i}=\epsilon_i,\\
0 & \ifc e_{k,n_i}=-\epsilon_i.
\end{array}
\right.$$
Then the number $A$ of sequences in $\sF$ satisfying (\ref{eqn.spec}) equals 
\be \label{eqn.number} \nn
A = \sum_{k=1}^{F}{\frac{1}{2^j} \prod_{i=1}^{j}{(1+\epsilon_ie_{k,n_i})}}.
\ee
Extending the product and after easy calculations we have
\be \nn
\begin{array}{lll}
A 
&=&\dd \frac{1}{2^j}\sum_{k=1}^{F}{\left[1+ \sum_{{\ell}=1}^{j}{\sum_{1 \leq  i_1 < \cdots < i_{\ell} \leq j}{\epsilon_{i_1}\cdots \epsilon_{i_{\ell}}e_{k,n_{i_1}} \cdots e_{k,n_{i_{\ell}}}}}\right]}\\[.6cm]&=&\dd \frac{1}{2^j}\left[F+ \sum_{{\ell}=1}^{j}{\sum_{1 \leq i_1 < \cdots < i_{\ell} \leq j}{\epsilon_{i_1}\cdots \epsilon_{i_{\ell}} \sum_{k=1}^{F}{e_{k,n_{i_1}} \cdots e_{k,n_{i_{\ell}}}}}} \right]\\[.6cm]
&\geq &\dd \frac{1}{2^j}\left[F- \sum_{{\ell}=1}^{j}{\sum_{1 \leq i_1 < \cdots < i_{\ell} \leq j}{\left|\sum_{k=1}^{F}{e_{k,n_{i_1}} \cdots e_{k,n_{i_{\ell}}}}\right|}} \right]\\[.6cm]
&\geq &
\dd \frac{1}{2^j}\left[F- \sum_{{\ell}=1}^{j}{\binomial{j}{{\ell}}{\Phi_{\ell}(\overline{\sF})}} \right].
\end{array}
\ee
Thus if 
\be \label{eqn.j2}
\dd F > \sum_{{\ell}=1}^{j}{\binomial{j}{{\ell}}{\Phi_{\ell}(\overline{\sF})}},
\ee
then there exists at least one sequence in $\sF$ satisfying (\ref{eqn.spec}). 

By (\ref{eqn.bound_f}) we may assume $j \leq \log_2F$, and so we get
\be \nn
\begin{array}{lll}
\dd \sum\limits_{\ell=1}^{j}{\binomial{j}{{\ell}}{\Phi_{\ell}(\overline{\sF})}} & \leq &   
 2^j \max\limits_{1 \leq i \leq \log_2 F}\Phi_{i}(\overline{\sF}).
\end{array}
\ee
Therefore for all integers $j \geq 0$ satisfying $$ j < \log_2{F} - \log_2{\max_{1 \leq \ell \leq \log_2{F}}{\Phi_{\ell}(\overline{\sF})}}$$
 (\ref{eqn.j2}) is satisfied and we have $A > 0$ which completes the proof.
\end{proof}

\section{A family with low cross-correlation measure and high family complexity}\label{sec.example}
In this section we demonstrate how to apply Theorem~\ref{thm.f_ccm} and prove that the following family of sequences and its dual family have both high family complexity and small cross-correlation
measure of order $\ell$ up to a large order $\ell$. 



Let $p>2$ be a prime and $b$ a quadratic nonresidue modulo $p$. 
The family $\sF$ and its dual family $\overline{\sF}$ are  defined as follows:
 $${\sF}=\left\{ \left(\frac{n^2-bi^2}{p}\right)_{i=1}^{(p-1)/2} : n=1,\ldots,(p-1)/2\right\},$$
 $$\overline{\sF}=\left\{ \left(\frac{n^2-bi^2}{p}\right)_{n=1}^{(p-1)/2} : i=1,\ldots,(p-1)/2\right\}.$$
 Since $\left(\frac{n^2-bi^2}{p}\right)=\left(\frac{-b}{p}\right)\left(\frac{i^2-bn^2}{p}\right)=(-1)^{(p+1)/2}\left(\frac{i^2-bn^2}{p}\right)$
 both families have the same family complexity and cross-correlation measure of order $\ell$. 
We will show
\begin{equation}\label{Ql}
 Q_\ell({\sF})\ll \ell p^{1/2}\log p\quad\mbox{ and }\quad
Q_\ell(\overline{\sF})\ll \ell p^{1/2}\log p
\end{equation}
for each integer $\ell=1,2,\ldots$ 
Then Theorem~\ref{thm.f_ccm} immediately implies
$$C({\sF})\ge \left(\frac{1}{2}-o(1)\right) \frac{\log p}{\log 2} \quad \mbox{and}\quad
C(\overline{\sF})\ge \left(\frac{1}{2}-o(1)\right) \frac{\log p}{\log 2}.$$
Here we used 
$$|\sF|=|\overline{\sF}|=(p-1)/2,\quad p\ge 11$$
which can be verified in the following way.
For some $1\le n_1<n_2\le (p-1)/2$ assume that $\left(\frac{n_1^2-bi^2}{p}\right)=\left(\frac{n_2^2-bi^2}{p}\right)$ for $i=1,\ldots,(p-1)/2$.
Since $i^2\equiv (p-i)^2 \bmod p$ and $\left(\frac{n_1^2}{p}\right)=\left(\frac{n_2^2}{p}\right)=1$ these Legendre symbols are the same for all $i=0,\ldots,p-1$
and thus
$$p=\sum_{i=0}^{p-1} \left(\frac{n_1^2-bi^2}{p}\right)\left(\frac{n_2^2-bi^2}{p}\right)\le 3p^{1/2}$$
by Weil's bound and we get a contradiction to $p\ge 11$. 

Now we prove (\ref{Ql}).
According to Definition \ref{def.ccm} and since the Legendre symbol is multiplicative we need to estimate sums of the form
$$
\dd \left \lvert\sum_{n=1}^M\left( \frac{h(n)}{p} \right)\right\rvert$$
where 
$$h(X)=((X+d_1)^2-bi_1^2)((X+d_2)^2-bi_2^2)\cdots ((X+d_\ell)^2-bi_\ell^2).$$ 
The result follows from the Weil bound after reducing these incomplete character sums to complete ones using the standard method provided that $h(X)$ is not a square, 
see \cite{iwko,ti,wi}. 
Since each factor of $h(X)$ of the form $(X+d)^2-bi^2$ is irreducible over $\F_p$, it is enough to show that one factor is distinct from all the others.
We may assume that the first factor is of the form $X^2-bi^2$. Comparing coefficients we see that
$$X^2-bi^2=(X+d)^2-bj^2$$ 
is only possible if $d=0$ and $i\equiv \pm j\bmod p$. Since $1\le i,j\le (p-1)/2$ we get $i=j$ and $h(X)$ is not a square since either $d_j\ne d_k$ for $i\ne k$ or $d_j=d_k$ and $i_j\ne i_k$
with $1\le i_j,i_k\le (p-1)/2$. 

\subsubsection*{Remarks}
\begin{enumerate}
\item[1.] Gyarmati \cite{Gya2009} studied the family $\sF$ of binary sequences $E_f = (e_{f,1}, \ldots, e_{f,p}) \in \{-1, +1\}^p$ defined by 
\be \label{Leg_seq}
e_{f,n} = \left\lbrace \begin{array}{ll}\left( \frac{f(n)}{p} \right) & \ifc \gcd(f(n),p) =1,\\
\quad 1 & \ifc p \mid f(n),
\end{array}
\right.
\ee
where $f$ runs through all non-constant square-free polynomials over $\F_p$ of degree at most $k$. She proved 
$$C(F) \geq \frac{K}{2\log{2}}\log{p} - O(k \log{(k \log{p})}).$$
However, this family has obviously a very large cross-correlation measure of order 2. (Compare $E_f$ and $E_g$ with $g(X) = f(X+d)$ for some $d \in \F_p^*$.)

\item[2.] Gyarmati, Mauduit, and S\'arközy \cite{GMS2004} studied the family $\sF$ of sequences $E_f = (e_{f,1}, \ldots, e_{f,p}) \in \{-1, +1\}^p$ defined by \eqref{Leg_seq} where $f$ runs through all monic irreducible polynomials of degree d with second leading coefficient 0 and showed in Theorem 8.14 that 
$$\Phi_l(\sF) \ll \ell d p^{1/2}\log{p}.$$
We note that for $d=2$ the polynomials $f$ are of the form $X^2 - a$ where $a$ is a quadratic non-residue modulo $p$ and we have 
$e_{f,1} = e_{f,p-1}$ and thus $C(\sF) \leq 1$. For $d \geq 3$ it seems to be a challenging problem to estimate the family complexity of this family.

\item[3.] Very recently, Gyarmati \cite{Gya2014ffa} proved also a lower bound on the family complexity of the family of sequences $E_f$ defined by \eqref{Leg_seq} where $f$ runs through \textit{all} monic irreducible polynomials of degree $d$ (with arbitrary second leading coefficient). However, since $f(X)$ is irreducible whenever $f(X+c)$ for all $c \in \F_p$ the cross-correlation measure of order 2 of this family is again very large.
\end{enumerate}

\section*{Acknowledgment}
The first author is supported by the Austrian Science Fund (FWF): Project F5511-N26 
which is part of the Special Research Program "Quasi-Monte Carlo Methods: Theory and Applications".
The second author is supported by T\"{U}B\.{I}TAK under Grant No.\ 2219. 

\bibliographystyle{amsplain}
\bibliography{family_complexity}

\end{document}